\input ./style/arxiv-vmsta.cfg
\documentclass[numbers,compress,v1.0.1]{vmsta}
\usepackage{vtexbibtags}
\usepackage{authorquery}

\volume{4}% Updated by VTEXPTS2LaTeX.exe, 28.12.2017 16:29
\issue{4}% Updated by VTEXPTS2LaTeX.exe, 28.12.2017 16:29
\pubyear{2017}
\firstpage{407}% Updated by VTEXPTS2LaTeX.exe, 28.12.2017 16:29
\lastpage{425}% Updated by VTEXPTS2LaTeX.exe, 28.12.2017 16:29
\aid{VMSTA92}
\doi{10.15559/17-VMSTA92}% Updated by VTEXPTS2LaTeX.exe, 12.12.2017
\articletype{research-article}

\startlocaldefs

\allowdisplaybreaks

\gdef\nz{{\mathbb N}} % positive integers
\gdef\rz{{\mathbb R}} % real numbers

\newcommand{\nm}{\{1,\ldots,n\}}

\newtheorem{theorem}{Theorem}
\newtheorem{proposition}{Proposition}
\newtheorem{lemma}{Lemma}
\newtheorem{remark}{Remarks}
\newtheorem{remarksingle}[remark]{Remark}

\newtheorem{corollary}{Corollary}

\theoremstyle{plain}% Theorems
\theoremstyle{definition} %Proclaim
\theoremstyle{remark}% Remark
\endlocaldefs

\begin{document}

\begin{frontmatter}
\pretitle{Research Article}

\title{On the size of the block of 1 for \texorpdfstring{$\varXi
$}{Xi}-coalescents with dust}

%\author[]{\inits{}\fnms{}~\snm{}\thanksref{f1}\ead[label=e1]{}}
%\author[]{\inits{}\fnms{}~\snm{}\thanksref{f1}\thanksref{cor1}
%\ead[label=e2]{}}
%\thankstext[type=corresp,id=cor1]{Corresponding author.}
%\address[]{\institution{}, ..., \cny{}}
%\address[]{\institution{}, ..., \cny{}}

%\thankstext[id=f1]{}

%\dedicated{}

\author[a]{\inits{F.}\fnms{Fabian}~\snm{Freund}\thanksref{cor1}\ead[label=e1]{fabian.freund@uni-hohenheim.de}}
\thankstext[type=corresp,id=cor1]{Corresponding author.}
\author[b]{\inits{M.}\fnms{Martin}~\snm{M\"ohle}\ead[label=e2]{martin.moehle@uni-tuebingen.de}}
%\fntext[f1]{Some remarks.}\fnref{f1}

\address[a]{Crop Plant Biodiversity and Breeding Informatics Group (350b),\\
Institute of Plant Breeding, Seed Science and Population Genetics,\\
\institution{University of Hohenheim},\\
Fruwirthstrasse 21, 70599 Stuttgart, \cny{Germany}}
\address[b]{Mathematical Institute,\\
\institution{Eberhard Karls University of T\"ubingen},\\
Auf der Morgenstelle 10, 72076 T\"ubingen, \cny{Germany}}

\markboth{F. Freund, M. M\"ohle}{On the size of the block of 1 for $\varXi
$-coalescents with dust}

%\markboth{Authors}{Title}
%\markboth{}{}
%
\begin{abstract}
We study the frequency process $f_1$ of the block of 1 for a
$\varXi$-coalescent $\varPi$ with dust. If $\varPi$ stays infinite, $f_1$
is a jump-hold process which can be expressed as a sum of broken
parts from a stick-breaking procedure with uncorrelated, but in
general non-independent, stick lengths with common mean. For
Dirac-$\varLambda$-coalescents with $\varLambda=\delta_p$,
$p\in[\frac{1}{2},1)$, $f_1$ is not Markovian, whereas its jump
chain is Markovian. For simple $\varLambda$-coalescents the
distribution of $f_1$ at its first jump, the asymptotic frequency
of the minimal clade of 1, is expressed via conditionally
independent shifted geometric distributions.
\end{abstract}

\received{\sday{28} \smonth{8} \syear{2017}}% Updated by
%VTEXPTS2LaTeX.exe, 12.12.2017 14:27
\revised{\sday{4} \smonth{12} \syear{2017}}% Updated by
%VTEXPTS2LaTeX.exe, 12.12.2017 14:27
\accepted{\sday{6} \smonth{12} \syear{2017}}% Updated by
%VTEXPTS2LaTeX.exe, 12.12.2017 14:27
\publishedonline{\sday{27} \smonth{12} \syear{2017}}
\end{frontmatter}

\section{Introduction and results}
Independently introduced in \cite{Schweinsberg2000} and
\cite{Moehle2001}, $\varXi$-coalescents are exchangeable Markovian
processes $\varPi=(\varPi_t)_{t\geq0}$ on the set of partitions of
$\nz:=\{1,2,\ldots\}$ whose transitions are due to mergers of
partition blocks. The distribution of $\varPi$ is characterised by a
finite measure $\varXi$ on the infinite simplex
\[
\varDelta:=\bigl\{x=(x_1,x_2,\ldots):x_1\geq
x_2\geq\cdots\geq0, |x|\leq1\bigr\},
\]
where $|x|:=\sum_{i\in\nz}x_i$. We exclude $\varXi=0$,
since it leads to a coalescent without coalescence events.
$\varXi$-coalescents allow that disjoint subsets of
blocks merge into distinct new blocks, hence they are
also called coalescents with simultaneous multiple mergers.
If $\varXi$ is concentrated on $[0,1]\times\{0\}\times\{0\}\times\cdots$,
only a single set of blocks is allowed to merge. Such
a coalescent is a $\varLambda$-coalescent, see \cite{Pitman1999}.
In this case, $\varLambda$ is a finite measure on $[0,1]$,
the restriction of $\varXi$ on the first coordinate of $\varDelta$.
The restriction $\varPi^{(n)}$ of $\varPi$ on $[n]:=\nm$ is called the
$\varXi$-$n$-coalescent. Denote the blocks of $\varPi_t$
by $(B_i(t))_{i\in\nz}$, where $i$ is the least element
of the block (we set $B_i(t)=\emptyset$ if $i$ is not a least element
of a block). Clearly, $1\in B_1(t)$. We call $B_1(t)$ the block of
1 at time $t$. Due to the exchangeability of the $\varXi$-coalescent,
Kingman's correspondence ensures that, for every $t\ge0$, the
asymptotic frequencies
\begin{equation}
\label{def:asymptfreq} f_i(t):=\lim_{n\to\infty}\frac{|B_i(t)\cap[n]|}{n},
\quad i\in\nz,
\end{equation}
exist almost surely, where $|A|$ denotes the cardinality of the set $A$.

The family of $\varXi$-coalescents is a diverse class of
processes with very different properties, see e.g. the review
\cite{gnedin2014lambda} for $\varLambda$-coalescents. We will focus on
$\varXi$-coalescents with dust, i.e. $\varXi$ fulfils (see \cite{Schweinsberg2000})
\begin{equation}
\label{def:dust} \mu_{-1}:=\int_{\varDelta} |x|
\nu_0(dx)<\infty,
\end{equation}
where $\nu_0(dx)=\varXi(dx)/(x,x)$ with
$(x,x):=\sum_{i\in\nz} x_i^2$ for $x=(x_1,x_2,\ldots)\in\varDelta$.
These coalescents are characterised by a non-zero probability that,
at any time $t$, there is a positive fraction of $\nz$, the dust,
that has not yet merged. Note that $i\in\nz$ is part of
the dust at time $t$ if and only if $\{i\}$ is a block at time $t$,
which is called a singleton block. The asymptotic
frequency of the dust component is
$S_t:=1-\sum_{i\in\nz} f_i(t)$. Having dust is equivalent to
$P(S_t>0)>0$ for all $t>0$.
We are interested in $\varXi$-coalescents which stay infinite, i.e. which
almost surely have an infinite number of blocks for each $t>0$. We will
put some further emphasis on simple $\varLambda$-coalescents satisfying
\begin{equation}
\label{def:simpleLambda} \mu_{-2}:=\int_{[0,1]}x^{-2}
\varLambda(dx)<\infty.
\end{equation}
This class includes Dirac coalescents with $\varLambda=\delta_p$, the
Dirac measure in $p\in(0,1]$.
Consider the frequency process $f_1:=(f_1(t))_{t\geq0}$ of the block
of 1. For $\varLambda$-coalescents, Pitman characterises $f_1$ as follows
(reproduced from \cite{Pitman1999}, adjusted to our notation).
\begin{proposition}\cite[Proposition 30]{Pitman1999}\label{prop:pitm30}
No matter what $\varLambda$, the process $f_1$ is an increasing pure jump
process with c\`adl\`ag paths,
$f_1(0)=0$ and $\lim_{t\to\infty}f_1(t)=1$.
If $\mu_{-1}=\infty$ then almost surely $f_1(t)>0$ for all $t>0$ and
$\lim_{t\searrow0}f_1(t)=0$. If
$\mu_{-1}<\infty$ then $f_1$ starts by holding at zero until an
exponential time with rate $\mu_{-1}$, when it enters $(0,1]$ by a
jump, and proceeds thereafter by
a succession of holds and jumps, with holding rates bounded above by
$\mu_{-1}$.
\end{proposition}
%
%Essentially, Proposition \ref{prop:pitm30} directly expands to $
%\varXi$-coalescents.
Moreover, in \cite[Section~3.9]{Pitman1999}, a general formula for the
moments of $f_1(t)$ for fixed $t>0$ is provided.

For two particular coalescents without dust, further properties of
$f_1$ are known. For Kingman's $n$-coalescent ($\varLambda=\delta_0$), the
complete distribution of block sizes is explicitly known, see \cite
[Theorem 1]{Kingman1982b}, from which one can derive some properties of
the block of 1 due to exchangeability. For the Bolthausen--Sznitman
coalescent ($\varLambda$ the uniform distribution on $[0,1]$) the block of
1 can be characterised as in \cite[Corollary 16]{Pitman1999}.
For instance, $f_1$ is Markovian for the Bolthausen--Sznitman coalescent.

Different specific aspects of the block of 1 have been analysed for
different $\varLambda$/$\varXi$-$n$-coalescents including their asymptotics
for $n\to\infty$.
\begin{itemize}
\item External branch length: The waiting time for the first jump of
the block of 1 in the $n$-coalescent, see e.g.
\cite{Blum2005,Caliebe_2007,Dhersin20131691,freund2009time,janson2011total,M_hle_2010}.
\item Minimal clade size: The size $M_n$ of the block of 1 for the
$n$-coalescent at its first jump. For Kingman's $n$-coalescent and for
$\varLambda$ beta-distributed with parameters $(2-\alpha,\alpha)$ with
$\alpha\in(1,2)$, $X_n$ converges in distribution for $n\to\infty$, see
\cite{Blum2005} and \cite{siri2016asymptotics}. For the
Bolthausen--Sznitman $n$-coalescent, $\log(M_n)/\log(n)$ converges in
distribution \cite{Freund2014}. These results do not cover $\varLambda
$/$\varXi$-coalescents with dust.
\item The number of blocks involved in the first merger of the block of
1, see \cite{siri2016asymptotics}. The results cover $\varLambda
$-coalescents with dust. %\cite{delmas2008asymptotic},
%\cite{Dhersin20131691},
%
\item The number of blocks involved in the last merger of the block of
1, see \cite{abraham2013construction,abraham2015beta,henard2015fixation,goldschmidt2005random,2017arXiv170100549K,mohle2014hitting}.
\item The small-time behaviour of the block of 1, see \cite
{berestycki2008small,siri2016asymptotics}.
\end{itemize}
Due to the exchangeability of the $\varXi$-coalescent, any result
for the distribution of the block of 1 holds true for the block containing
any other $i\in\nz$. We want to further describe $f_1$ for $\varXi$-coalescents
with dust. For any finite measure $\varXi$ on $\varDelta$ which fulfils
$\eqref{def:dust}$, we introduce
\begin{equation}
\label{def:gamma} \gamma:=\frac{\varXi(\varDelta)}{\mu_{-1}}.
\end{equation}
%
%Eq. \eqref{def:dust} implies $\varXi(\{0\})=0$ and
We see that $\gamma\in(0,1]$, since
\[
0<\varXi(\varDelta)=\int_{\varDelta} (x,x) \nu_0(dx)\leq
\int_{\varDelta} |x| \nu _0(dx)=\mu_{-1}<
\infty.
\]
Define $\varDelta_f:=\bigcup_{k\in\nz}\{x\in\varDelta:x_1+\cdots+x_k=1\}$.
We extend Proposition \ref{prop:pitm30} for $\varXi$-coalescents with dust
which stay infinite, i.e.~have almost surely infinitely many blocks for each
$t\geq0$ (equivalent to $\varXi(\varDelta_f)=0$, see Lemma \ref{lem:CDI}).
While the extension to $\varXi$-coalescents and the explicit waiting time
distributions are a direct follow-up from Pitman's proof, we provide
a more detailed description of the jump heights of $f_1$. Proposition
\ref{prop:pitm30} ensures that the jumps of $f_1$ are separated by
(almost surely) positive waiting times, we denote the value of $f_1$ at
its $k$th jump with $f_1[k]$ for $k\in\nz$.
\begin{theorem}\label{thm:blockof1Xi}
In any $\varXi$-coalescent $\varPi$ with dust and $\varXi(\varDelta_f)=0$, the
asymptotic frequency process $f_1:=(f_1(t))_{t\geq0}$
of the block of $1$, defined by
Eq.~\eqref{def:asymptfreq}, is an increasing pure jump process with
c\`adl\`ag paths, $f_1(0)=0$ and $\lim_{t\to\infty}f_1(t)=1$,\vadjust{\goodbreak} but
$f_1(t)<1$ for $t>0$ almost surely. The waiting times between almost
surely infinitely many jumps are distributed as independent ${\rm
Exp}(\mu_{-1})$ random variables. Its jump chain $(f_1[k])_{k\in\nz}$
can be expressed via stick-breaking
\begin{equation}
\label{eq:stickbreak} f_1[k]=\sum_{i=1}^{k}X_i
\prod_{j=1}^{i-1}(1-X_j),
\end{equation}
where $(X_j)_{j\in\nz}$ are pairwise uncorrelated, $X_j>0$ almost
surely and $E(X_j)=\gamma$ for all $j\in\nz$. In particular,
$E(f_1[k])=1-(1-\gamma)^k$.
In general, $(X_j)_{j\in\nz}$ are neither independent nor identically
distributed.
\end{theorem}
\begin{remarksingle}
From Theorem \ref{thm:blockof1Xi}, the dependence between $f_1$ and its
jump times is readily seen as follows.
Recall \cite[Eq.~(51)]{Pitman1999} that
$E(f_1(t))=1-e^{-t}$ for any $\varLambda$-coalescent with $\varLambda([0,1])=1$.
If we would have independence, integrating
$E(f_1(t))$ over the waiting time distribution
${\rm Exp}(\mu_{-1})$ for the first jump of $f_1$
would yield $E(f_1[1])=(1+\mu_{-1})^{-1}$, in contradiction to
$E(f_1[1])=1/\mu_{-1}$ by Theorem \ref{thm:blockof1Xi}.
\end{remarksingle}
Dirac coalescents ($\varLambda=\delta_p$ for some $p\in(0,1]$)
are a family of $\varLambda$-coalescents with dust. They
have been introduced as simplified models
for populations in species with skewed offspring distributions
(reproduction sweepstakes), see \cite{Eldon2006}. Their jump chains
(discrete time Dirac coalescents)
can also arise as large population size limits in conditional
branching process models \cite[Theorem 2.5]{Huillet2015}.

We further characterise $f_1$ as follows,
including an explicit formula for its distribution at its first
jump.
\begin{proposition}\label{cor:Dirac}
Let $\varLambda=\delta_p$, $p\in[\frac{1}{2},1)$ and
$q:=1-p$. $f_1$ takes values in the set
\begin{equation}
\label{eq:unique} \mathcal{M}_p:=\Biggl\{\sum
_{i\in\nz} b_i pq^{i-1} : b_i
\in\{0,1\},1\leq\sum_{i\in\nz} b_i <
\infty\Biggr\}.
\end{equation}
For $x=\sum_{i\in\nz}b_i pq^{i-1}\in\mathcal{M}_p$, we have
\begin{equation}
\label{eq:distXDirac} P\bigl(f_1[1]=x\bigr)=pq^{j-1} \prod
_{i\in J\setminus\{j\}}P(Y+i\in J)\prod_{i\in[j-1]\setminus J} P(Y+i
\notin J)>0,
\end{equation}
where $Y\stackrel{d}{=}Geo(p)$, $J:=\{i\in\nz|b_i=1\}$ and $j:=\max J$.
The process $f_1$ is not Markovian
whereas its jump chain $(f_1[k])_{k\in\nz}$ is Markovian.
\end{proposition}
\begin{remark}\
\begin{itemize}
\item The law of $f_1[1]$ is a discrete
measure on $[0,1]$ for Dirac coalescents. Surprisingly
different properties arise for different values of $p$.
For instance, $\mathcal{M}_{2/3}=\{\mbox{$\sum_{i\in\nz}$}
b_i 3^{-i} : b_i\in\{0,2\},1\leq\mbox{$\sum_{i\in\nz}$}b_i < \infty\}$
is a subset of the ternary Cantor set which is nowhere dense in $[0,1]$,
whereas $\mathcal{M}_{1/2}$, the set of all $x\in[0,1]$ with finite
2-adic expansion, is dense in $[0,1]$.
\item We omitted $f_1[1]$ for the star-shaped coalescent
($\varLambda=\delta_1$), since it just jumps from $0$ to $1$
at time $T\stackrel{d}{=}{\rm Exp}(1)$.
\item Recall that $f_1$ is Markovian for the Bolthausen--Sznitman
coalescent in contrast to $f_1$ for the Dirac coalescents specified above.
\end{itemize}
\end{remark}
Our key motivation was to provide a more detailed description of the
jump chain of $f_1$, especially properties of the value $f_1[1]$ at the
first jump which is the asymptotic frequency of the minimal clade.
Theorem \ref{thm:blockof1Xi} provides a first-order limit result for
all $\varXi$-coalescents with dust.
\begin{corollary}\label{cor:mcs}
Let $\varPi$ be a $\varXi$-coalescent with dust and $\varPi^{(n)}$ its restriction
on $[n]$. Let $M_n$ be the minimal clade size, i.e.~ the
size of the block of 1 at its first merger in $\varPi^{(n)}$. Then,
$M_n/n\to f_1[1]$ almost surely, $f_1[1]>0$ almost surely and
$E(f_1[1])=\gamma$.
\end{corollary}
Compared to the known results listed above for the minimal clade size for
dust-free coalescents, the minimal clade size is much larger
asymptotically for $n \to\infty$ ($O(n)$ compared to $o(n)$).

The law of $f_1[1]$ in (\ref{eq:distXDirac})
follows from the following more general description of $f_1[1]$ for
simple $\varLambda$-coalescents.
We introduce, for a finite measure $\varLambda$ on $[0,1]$ with $\mu
_{-1}=\int_0^1 x^{-1} \varLambda(dx)<\infty$,
\begin{equation}
\label{eq:alpha} \alpha:=\frac{\mu_{-1}}{\mu_{-2}}=\frac{\int_0^1 x^{-1} \varLambda
(dx)}{\int_0^1 x^{-2} \varLambda(dx)}.
\end{equation}
We have $\alpha\in[0,1]$ since $x^{-1}\leq x^{-2}$ on $(0,1]$ (if $\mu
_{-1}<\infty$, we have $\varLambda(\{0\})=0$). Additionally, $\alpha>0$ if
and only if $\mu_{-2}<\infty$, so if $\varLambda$ characterises a simple
coalescent (recall that $\mu_{-2}\geq\mu_{-1}>0$ since we exclude
$\varLambda=0$).
\begin{proposition}\label{prop:f1simple}
Let $\varLambda$ fulfil \eqref{def:simpleLambda}. Then,
\begin{equation}
\label{eq:dist_C} f_1[1]=\sum_{i=1}^C
B^{(C)}_i P_i\prod
_{j\in[i-1]}(1-P_j)= \sum_{i\in
\nz}P_i
\prod_{j\in[i-1]}(1-P_j)\sum
_{k\geq i}B_i^{(k)}1_{\{C=k\}},
\end{equation}
where $(P_i)_{i\in\nz}$ are i.i.d.~with $P_i\stackrel{d}{=}\mu
_{-2}^{-1}x^{-2}\varLambda(dx)$. We have
\[
P\bigl(C=k\rvert(P_i)_{i\in\nz}\bigr)=P_k\prod
_{j\in[k-1]}(1-P_j), \mbox{ $C$ is $Geo(
\alpha)$-distributed.}
\]
Given $(P_i)_{i\in\nz}$, $C$ and $(B^{(k)}_i)_{k\in\nz,i\in[k]}$ are
independent and $(B^{(k)}_i)_{k\in\nz,i\in[k]}$ is defined as
\begin{align}
\label{eq:jointdistbi} & P\bigl(\bigl(B^{(j)}_1,
\ldots,B^{(j)}_{j}\bigr)=b|(P_i)_{i\in\nz}
\bigr)
\nonumber
\\
&\quad= \prod_{i\in J\setminus\{j\}}P\bigl(I(i)\in J|(P_i)_{i\in\nz}
\bigr)\prod_{i\in
[j-1]\setminus J} P\bigl(I(i) \notin
J|(P_i)_{i\in\nz}\bigr) \ \ \mbox{ almost surely},
\end{align}
where $b=(b_1,\ldots,b_j)\in\{0,1\}^{j-1}\times\{1\}$, $J:=\{i\in
[j]|b_i=1\}$ and, for each $i\in\nz$, $I(i):=\min\{j\geq
i+1:B^{(j)}_i=1\}$.
We have
\begin{itemize}
\item[(i)] $P(I(i)=i+k|(P_i)_{i\in\nz})=P_{i+k}\prod_{l=i+1}^{i+k-1}(1-P_l)$ almost surely for $k\in\nz$.
\item[(ii)] For any $i\in\nz$, $I(i)-i$ is $Geo(\alpha)$-distributed on
$\nz$. Given $(P_i)_{i\in\nz}$, $(I(i))_{i\in\nz}$ are independent.\vadjust{\goodbreak} %If
%$\varLambda=\delta_p$ for $p\in(0,1]$, $(I(i)-i)_{i\in\nz}$ are
%i.i.d.~$Geo(p)$-distributed.
\end{itemize}
\end{proposition}
\begin{remark}\leavevmode
\begin{itemize}
\item The distribution of $C$ is known from \cite[Proposition 3.1]{Gnedin2008}.
\item The distribution of $f_1[1]$ for Dirac coalescents with
$p>\frac{1}{2}$ has a structure somewhat similar to
the Cantor distribution, see e.g.~\cite{Lad1992} and \cite{Grabner1996}.
The Cantor distribution is the law of
$\sum_{i\in\nz}B_ipq^{i-1}$ for
$p\in(0,1)$, where $(B_i)_{i\in\nz}$ are
i.i.d.~Bernoulli variables with success probability $\frac{1}{2}$,
whereas in our case $(B_i)_{i\in\nz}$ are dependent
Bernoulli variables with success probabilities
$P(B_i=1)=P(\sum_{k\geq i} B^{(k)}_i 1_{\{C=k\}}=1)=pq^{i-1}+\sum_{k>i}p^2q^{k-1}
=pq^{i-1}(1+q)$, see Eq.~\eqref{eq:dist_C}. The Cantor distribution is
a shifted infinite Bernoulli convolution. Infinite Bernoulli
convolutions are the set of distributions of $\sum_{i\in\nz} \omega_i
(-1)^{B_i}$ with $\omega_i\in\rz$ for $i\in\nz$ satisfying $\sum_{i\in
\nz} \omega_i^2<\infty$, see \cite[Section~2]{peres2000sixty}. They
have been an active field of research since the 1930's, e.g. see \cite
{erdos1939family,solomyak1995random} and the survey \cite
{peres2000sixty}.
\end{itemize}
\end{remark}
Our main tool for the proofs is Schweinsberg's
Poisson construction of the $\varXi$-coalescent. The
article is organised as follows. We
recall (properties of) the Poisson construction in Section~\ref{sect:PoissC}. Section~\ref{sect:CDI}
characterises staying infinite for $\varXi$-coalescents with dust. These
prerequisites are then used to prove the results
for $\varXi$-coalescents with dust in Section~\ref{sect:dust}
and for simple $\varLambda$-coalescents in Section~\ref{sec:simple}.
\section{Poisson construction of a \texorpdfstring{$\varXi
$}{Xi}-coalescent and the block of 1}\label{sect:PoissC}
We recall the construction of a $\varXi$-$n$-coalescent $\varPi$ from
\cite{Schweinsberg2000}. We are only interested in constructing a $\varXi
$-coalescent with dust, which implies $\varXi(\{0\})=0$, see Eq. \eqref{def:dust}.

Let $\mathcal{P}$ be a Poisson point process on $A=[0,\infty)\times\nz
_0^{\infty}$ with intensity measure
\begin{equation}
\label{def:PPPintmeas} \nu=dt \otimes\int_\varDelta\otimes_{n\in\nz}
P^{(x)} \nu_0(dx),
\end{equation}
where, for $x\in\varDelta$, $P^{(x)}$ is a probability
measure on $\nz_0$ with $P^{(x)}(\{k\})=x_k$ and $P^{(x)}(\{0\})=1-|x|$
(Kingman's paintbox) and $\nu_0$ is defined as in Eq.~\eqref{def:dust}.
For $n\in\nz$, the restriction $\varPi^{(n)}$ of $\varPi$ to $[n]$ can be
constructed by starting at $t=0$ with each $i\in[n]$ in its own block.
Then, for each subsequent time $(T=)t$ with a Poisson point
$(T,(K_i)_{i\in\nz})$, merge all present blocks $i$ (at most $n$)
with identical $k_i>0$, where $i$ is the least element of the block
(there are only finitely many points of $\mathcal{P}$ that lead to a
merger of blocks in $[n]$). $\varPi$ is then pathwise defined by its
restrictions $(\varPi^{(n)})_{n\in\nz}$. From now on we will
assume without loss of generality that the $\varXi$-coalescent with dust
is constructed via the Poisson process $\mathcal{P}$.

The block of 1 can only merge at Poisson points $P=(T,(K_i)_{i\in\nz})$
with $K_1>0$. We take a closer look at these Poisson points. We
introduce exchangeable($Q$) indicators following \cite
[p.1884]{Pitman1999}: These are exchangeable Bernoulli variables which
are conditionally i.i.d.~given a random variable $X$ with distribution
$Q$ on $[0,1]$ which gives their success probability. Alternatively, we
denote these as exchangeable($X$) indicators if we can specify $X$.
\begin{lemma}\label{lem:PPPof1}
For any finite measure $\varXi$ on $\varDelta$ fulfilling \eqref{def:dust},
$\mathcal{P}$ splits into two independent Poisson processes\vadjust{\goodbreak}
\[
\mathcal{P}_1:=\bigl\{\bigl(T,(K_i)_{i\in\nz}
\bigr)\in\mathcal{P}: K_1>0\bigr\} \quad\mbox{and}\quad \mathcal{P}_2:=
\bigl\{\bigl(T,(K_i)_{i\in\nz}\bigr)\in\mathcal{P}:
K_1=0\bigr\}.
\]
$\mathcal{P}_1$ has almost surely finitely many points on any set
$[0,t]\times\nz_0^{\infty}$, thus we can order
\[
\mathcal{P}_1=\bigl(\bigl(T_j,\bigl(K^{(j)}_i
\bigr)_{i\in\nz}\bigr)\bigr)_{j\in\nz},
\]
where $T_j<T_{j+1}$ almost surely for $j\in\nz$.

$(T_j)_{j\in\nz}$ is a homogeneous Poisson process on $[0,\infty)$ with
intensity $\mu_{-1}$.

$((K^{(j)}_i)_{i\in\nz})_{j\in\nz}$ is an i.i.d.~sequence in $j$ and
$(1_{\{K^{(1)}_1=K^{(1)}_i\}})_{i\geq2}$ are exchange\-able($Q$)
indicators with
\[
Q:=\frac{1}{\mu_{-1}}\int_{\varDelta}\sum
_{i\in\nz}x_i\delta_{x_i}
\nu_0(dx),
\]
which is a probability measure on $[0,1]$. For $X\stackrel{d}{=}Q$, we
have $X>0$ almost surely and $E(X)=\gamma$.
\end{lemma}
\begin{proof}
$\mathcal{P}_1$ and $\mathcal{P}_2$
are obtained by restricting $\mathcal{P}$ on the
disjoint subsets $A_1:=[0,\infty)\times\nz\times\nz_0^\infty$ and
$A_2:=[0,\infty)\times\{0\} \times\nz_0^{\infty}$ of $A$. Thus,
$\mathcal{P}_1$ and $\mathcal{P}_2$ are independent Poisson processes
(restriction theorem \cite[p.17]{Kingman1993}) with intensity measures
$\nu_1=\nu(\cdot\cap A_1)$ and $\nu_2=\nu(\cdot\cap A_2)$.
For any Borel set $B\subseteq[0,\infty)$ and $\lambda$ being the
Lebesgue measure,
\begin{equation}
\label{eq:PPPprodmeas} \nu_1\bigl(B\times\nz_0^\infty
\bigr) =\lambda(B)\int_\varDelta\underbrace{P^{(x)}(
\nz)}_{=|x|}\prod_{i\ge
2}\underbrace{P^{(x)}(
\nz_0)}_{=1}\nu_0(dx)=\lambda(B)
\mu_{-1}.
\end{equation}
Thus, on any bounded set $B$, $\mathcal{P}_1$ has almost surely
finitely many points, which can be ordered as described.
Projecting $\mathcal{P}_1$ on the first coordinate $t$
of $A$ yields a Poisson process with intensity measure $\mu_{-1}dt$
(mapping theorem \cite[p.18]{Kingman1993}).

Now, we project the points of $\mathcal{P}_1$ on the coordinate of
$(K^{(j)}_i)_{i\in\nz}$. Recall the construction of a Poisson process
as a collection of i.i.d.~variables with distribution $(\mu(C))^{-1}\mu$
on sets of finite mass $C$ of the intensity measure $\mu$,
e.g.~\cite[p.23]{Kingman1993}. It shows that we can treat the
collection of $(T_j,(K^{(j)}_i)_{i\in\nz})$ with, for instance,
$T_i\in[k,k+1)$ for $k\in\nz$ as a random number of i.i.d.~random
variables with distribution $(1\cdot\mu_{-1})^{-1}\nu_1$. Since
$\nu_1$ has a product structure on $A_1$, we have that
$((K^{(j)}_i)_{i\in\nz})_{j\in\nz}$ are i.i.d.~in $j$ and have
distribution, for $m\in\nz$,
\begin{equation}
\label{eq:PPPkidist} P\bigl(\bigl(K_i^{(1)}=l_i
\bigr)_{i\in[m]}\bigr) =\frac{1}{\mu_{-1}}\int_\varDelta
\prod_{i\in[m]} P^{(x)}(l_i)
\nu_0(dx) =\frac{1}{\mu_{-1}}\int_\varDelta\prod
_{i\in[m]} x_{l_i}\nu_0(dx)
\end{equation}
for $l_1\in\nz$ and $l_2,\ldots,l_m\in\nz_0$. Consider
$(1_{\{K^{(1)}_1=K^{(1)}_i\}})_{i\geq2}$. To show that they
are exchangeable($Q$) indicators, \cite[Eq.~(27)]{Pitman1999} has
to be fulfilled, i.e.~we need to show
$P(\{i\in[m]:K^{(1)}_i=K^{(1)}_1\}=M)=E(X^{|M|-1}(1-X)^{m-|M|})$
for $X\stackrel{d}{=}Q$ and any $M\subseteq[m]$ with $1\in M$.
Using Eq.~\eqref{eq:PPPkidist} we compute
\begin{align*}
P\bigl(\bigl\{i\in[m]:K^{(1)}_i=K^{(1)}_1
\bigr\}=M\bigr)&= \sum_{j\in\nz} P\bigl(\bigl\{i\in
[m]:K^{(1)}_i=j\bigr\}=M,K_1^{(1)}=j
\bigr)
\\
&= \frac{1}{\mu_{-1}}\int_\varDelta\sum
_{j\in\nz
}x_j^{|M|}(1-x_j)^{m-|M|}
\nu_0(dx)\\
&=E\bigl(X^{|M|-1}(1-X)^{m-|M|}\bigr).
\end{align*}
Clearly, $P(X>0)=1$ since $\varXi(\{0\})=0$ and $E(X)=\mu_{-1}^{-1}
\int_\varDelta(x,x)\nu_0(dx)=\gamma$.
\end{proof}
\begin{remark}\label{rem:PPP}\
\begin{itemize}
\item The properties of the exchangeable($Q$) indicators remind of
\cite[Lemma 21, Theorem 4] {Pitman1999} and
\cite[Proposition 6]{Schweinsberg2000}. Restricting
$\mathcal{P}$ to points with $K_1=K_2>0$ we can reproduce their results
analogously to the proof of Lemma \ref{lem:PPPof1}.
\item$Q$ can be seen as the expected value of the random probability
measure $Q_x:=|x|^{-1}\sum_{i\in\nz}x_i\delta_{x_i}$ for $x\in\varDelta$
with $x$ drawn from $\mu_{-1}^{-1}|x|\nu_0(dx)$. In the Poisson
construction, this means we draw a "paintbox" $x\in\varDelta$ and then
record in which box the ball of $1$ falls, if we only allow it to fall
in boxes $1,2,\ldots$.
\item Consider a simple $\varLambda$-coalescent.
Projecting $\mathcal{P}_2$ on its first component,
so\break$(T,(K_i)_{i\in\nz})\mapsto T$, yields a homogeneous
Poisson process with intensity $\mu_{-2}-\mu_{-1}<\infty$. To see this,
proceed analogously as for $\mathcal{P}_1$. Then, Eq.~\eqref{eq:PPPprodmeas}
for $\nu_2$ reads the same except for replacing
$P^{(x)}(\nz)$ by $P^{(x)}(\{0\})=1-|x|$.
\end{itemize}
\end{remark}
For a $\varLambda$-coalescent (with $\varLambda(\{0\})=0$) the Poisson
construction simplifies, since $\varXi$ only has mass on
$\{x\in\varDelta:x_2=x_3=\cdots=0\}$ and thus $\mathcal{P}$ can be
seen as a Poisson process on $[0,\infty)\times\{0,1\}^{\infty}$
with intensity measure $dt \otimes\int_{[0,1]}\otimes_{n\in\nz}
P^{(x)} x^{-2}\varLambda(dx)$, where $P^{(x)}$ is the Bernoulli
distribution with success probability
$x\in(0,1]$.

\querymark{Q1}When constructing simple $\varLambda$-coalescents, even the
process $\mathcal{P}$ itself has almost surely finitely many points
$(T_j,(K_i^{(j)})_{i\in\nz})$ on any set $[0,t]\times\{0,1\}^{\infty}$
(which we can again order in the first coordinate). As described in
\cite[Example 19]{Pitman1999} and analogously to Lemma \ref
{lem:PPPof1}, we can construct each (potential) merger at point
$(T_j,(K_i^{(j)})_{j\in\nz})$ of a simple $\varLambda$-coalescent as
follows (while between jumps, we wait independent\break ${\rm Exp}(\mu_{-2})$
times). First choose $P_i\in(0,1]$ from $\mu_{-2}^{-1}x^{-2}\varLambda
(dx)$, we have $E(P_i)=\break\mu_{-2}^{-1}\int_{[0,1]}x^{-1}\varLambda(dx)=\alpha
$. Then, throw independent coins $(K_i^{(j)})_{i\in\nz}$ with
probability $P_i$ for `heads' (=1) for each block present and merge all
blocks whose coins came up `heads'. Again, $(P_i)_{i\in\nz}$ are
i.i.d.~and the `coins' $K_i^{(j)}$ are exchangeable$(P_i)$ indicators.
Analogously to above, we thus have
\begin{lemma}\label{lem:simplePPP}
Let $\varLambda$ be a finite measure on $[0,1]$ fulfilling \eqref
{def:simpleLambda}. For the Poisson process $\mathcal
{P}=(T_j,(K_i^{(j)})_{i\in\nz})_{j\in\nz}$, $((K_i^{(j)})_{i\in\nz
})_{j\in\nz}$ is an i.i.d.~sequence (in $j$) of sequences of
exchangeable($P_j$) indicators, where $(P_j)_{j\in\nz}$ are i.i.d.~with
$P_1\stackrel{d}{=}\mu_{-2}^{-1}x^{-2}\varLambda(dx)$. In particular,
$E(P_i)=\alpha$.
%We have $$\lim_{n\to\infty}\frac{1}{n}\sum_{i=1}^n1_{\{K^{(1)}_1=1
%\}}>0$$ a.s..
%\begin{align*}
%& \lim_{n\to\infty}\frac{1}{n}\sum_{i=1}^n1_{\{K^{(1)}_1=1\}}>0 \mbox{
%a.s., }\\ & E\left(\lim_{n\to\infty}\frac{1}{n}\sum_{i=1}^n1_{
%\{K^{(1)}_1=1\}}\right)=P(K^{(1)}_1=1)= \mbox{ for } i\in\nz, i\neq1.
%\end{align*}
%for all $i\in\nz$
\end{lemma}

Since many proofs will build on the properties of different sets of
exchangeable indicators, we collect some well-known properties in the following

\begin{lemma}\label{lem:exch_ind}
Let $(K_i)_{i\in\nz}$ be exchangeable($X$) indicators.
\begin{itemize}
\item[a)] We have $\lim_{n\to\infty}\frac{1}{n}\sum_{i=1}^n K_i =X \mbox
{ almost surely}$. $X$ is almost surely unique.
\item[b)] If $(K_i)_{i\in\nz}$ is independent of a $\sigma$-field
$\mathcal{F}$, $X$ is, too.\vadjust{\goodbreak}
\item[c)] Let $(L_i)_{i\in\nz}$ be exchangeable($Y$) indicators,
independent of $(K_i)_{i\in\nz}$. Then,\break $(K_iL_i)_{i\in\nz}$ are
exchangeable($XY$) indicators and $X$, $Y$ are independent.
\end{itemize}
\end{lemma}
\begin{proof}
These properties essentially follow from the de Finetti representation
of an infinite series of exchangeable variables as conditionally
i.i.d.~variables. The lemma is a collection of well-known properties as
e.g. described in \cite[Sections~2 and 3]{aldous1985exchangeability},
arguments of which we use in the following.

An infinite exchangeable sequence is conditionally i.i.d.~given an
almost surely unique random measure $\alpha$. This measure is the weak
limit of the empirical measures, in our case, $n^{-1}\sum_{i=1}^{n}\delta_{K_i}$, which has limit $X'\delta_{1}+(1-X')\delta_0$
for some random variable $X'$ with values in $[0,1]$. Given $\alpha$, the
indicators are $\alpha$-distributed. However, since $X$ gives the
success probability of each Bernoulli coin, we have $X=X'$ almost
surely, so $X$ is almost surely unique. The rest of a) is just the
strong law of large numbers e.g. from \cite
[2.24]{aldous1985exchangeability} ($E(K_1)\leq1$), the limit is $X'$.
Part b) follows from measure theory since the limit is measurable in
the $\sigma$-field spanned by the summed variables. For c), we again
check Pitman's condition \cite[Eq.~27]{Pitman1999}. Let $M\subseteq
[m]$. We have that $X$, $Y$ are independent from b). With
$P(K_i=L_i=1|X,Y)=XY$ almost surely,
\begin{align*}
P\bigl(\bigl\{i\in[m]: K_iL_i=1\bigr\}=M\bigr)&=E
\bigl(P\bigl(\bigl\{i\in[m]: K_iL_i=1\bigr\}=M|X,Y\bigr)
\bigr)
\\
&=E\bigl((XY)^{|M|}(1-XY)^{m-|M|}\bigr),
\end{align*}
since given $X$, $Y$, both $(K_i)_{i\in\nz}$ and $(L_i)_{i\in\nz}$ are
independent. This shows c).
\end{proof}

\section{When does a \texorpdfstring{$\varXi$}{Xi}-coalescent with dust
stay infinite?}\label{sect:CDI}
\querymark{Q2}A crucial assumption for our results is that the $\varXi
$-coalescent $\varPi$ has almost surely infinitely many blocks that may
merge in the mergers where 1 participates in. The property
\[
P(\varPi_t \mbox{ has infinitely many blocks}\ \forall\ t>0)=1
\]
is called staying infinite, while $P(\varPi_t \mbox{ has finitely many
blocks}\ \forall\ t>0)=1$ is the property of coming down from infinity.
These properties have been thoroughly discussed for $\varXi$-coalescents,
see e.g. \cite{Schweinsberg2000,limic2010speed} and \cite
{herriger2012conditions}.

We recall the condition for $\varXi$-coalescents with dust to stay infinite.
\begin{lemma}\label{lem:CDI}
Let $\varDelta_f:=\{x\in\varDelta: x_1+\cdots+x_k=1 \mbox{ for some } k\in\nz
\}$
and $\varXi$ be a finite measure on $\varDelta$ fulfilling Eq.~\eqref{def:dust}.
The $\varXi$-coalescent stays infinite if and only if $\varXi(\varDelta_f)=0$.
If $\varXi(\varDelta_f)>0$, then the $\varXi$-coalescent has infinitely many
blocks until the first jump of $f_1$ almost
surely and the $\varXi$-coalescent neither comes down from infinity nor stays
infinite.
\end{lemma}
\begin{proof}
Let $\varDelta^*:=\{x\in\varDelta:|x|=1\}$. All $\varXi$-coalescents considered
are constructed via the Poisson construction with Poisson point process
$\mathcal{P}$.

First, assume $\varXi(\varDelta^*)=0$. We recall the (well-known) property
that for a $\varXi$-coalescent with dust $\varXi(\varDelta^*)=0$ is equivalent
to $P(S_t>0 \ \forall t)=1$, where $S_t$ is the asymptotic frequency of
the dust component. We use the remark on \cite
[p.1091]{freund2009number}: For $\varXi$-coalescents with dust, $(-\log
S_t)_{t\geq0}$ is a subordinator. The subordinator jumps to $\infty$
(corresponds to $S_t=0$) if and only if for its Laplace exponent $\varPhi
$, we have $\lim_{\eta\searrow0}\varPhi(\eta)>0$. For a $\varXi$-coalescent
with dust we have $\lim_{\eta\searrow0}\varPhi(\eta)=\int_{\varDelta^*} \nu
_0(dx)$. Hence, $\varXi(\varDelta^*)=0$ almost surely guarantees infinitely
many singleton blocks for all $t\geq0$, so the corresponding $\varXi$
coalescent stays infinite.

Now assume $\varXi(\varDelta^*)>0$. The subordinator $(-\log S_t)_{t\geq0}$
jumps from finite values ($S_t>0$) to $\infty$ ($S_t=0$) after an
exponential time with rate $\nu_0(\varDelta^*)$. This shows that the $\varXi
$-coalescent does not come down from infinity. Assume further that $\varXi
(\varDelta_f)=0$. Then, \cite[Lemma 31]{Schweinsberg2000} shows that the
$\varXi$-coalescent either comes down from infinity or stays infinite, so
it stays infinite.

Finally, assume $\varXi(\varDelta_f)>0$. Split $\mathcal{P}$ into independent
Poisson processes $\mathcal{P}'_1:=\{(T,(K_i)_{i\in\nz})\in\mathcal
{P}:\kappa\in\varDelta_f\}$ and $\mathcal{P}'_2:=\{(T,(K_i)_{i\in\nz})\in
\mathcal{P}:\kappa\notin\varDelta_f\}$, where $\kappa:=(\lim_{n\to\infty
}n^{-1}\sum_{i\in[n]} 1_{\{K_i=j\}})_{j\in\nz}$ (again restriction
theorem \cite[p.17]{Kingman1993}, Lemma \ref{lem:exch_ind} shows $\kappa
$ exists almost surely). Their intensity measures are defined as in
Eq.~\eqref{def:PPPintmeas}, but using $\nu'_1(\cdot):=\nu_0(\cdot\cap
\varDelta_f)$ and $\nu'_2:=\nu_0-\nu'_1$ instead of $\nu_0$. Since $\nu
'_1(\varDelta_f)\leq\mu_{-1}<\infty$, for any $t>0$ there are almost
surely finitely many $P\in\mathcal{P}'_1$ with $T<t$. Consider such
$P=(T,(K_i)_{i\in\nz})$ with $T$ smallest. Observe that until $T$, we
can construct the $\varXi$-coalescent using only the points of $\mathcal
{P}'_2$, which is the construction of a $\varXi'$-coalescent with $\varXi
'(dx):=(x,x)\nu'_2(dx)$. Since $\int_{\varDelta} |x|\nu'_2(dx)<\mu
_{-1}<\infty$ and $\varXi'(\varDelta_f)=0$, the proof steps above show that
the $\varXi$-coalescent has infinitely many blocks until $T$. Now consider
the merger at time $T$. The form of $\nu'_1$ ensures that $(K_i)_{i\in
\nz}$ can only take finitely many values, and Lemma \ref
{lem:exch_ind}a) ensures that infinitely many $K_i$'s show each value.
Thus, all blocks present before time $T$ are merged at $T$ into a
finite number of blocks (given by which $K_i$'s show the same number).
This shows that if $\varXi(\varDelta_f)>0$, the $\varXi$-coalescent stays
neither infinite nor comes down from infinity. Additionally, this shows
that either the block of 1 already merged at least once before $T$ or
it merges at $T$, thus there are infinitely many blocks before the
first merger of~1.
\end{proof}

\section{The block of 1 in \texorpdfstring{$\varXi$}{Xi}-coalescents with
dust -- proofs and remarks}\label{sect:dust}
\begin{proof}[Proof of Theorem \ref{thm:blockof1Xi}]
As in Lemma \ref{lem:PPPof1}, split the Poisson point process $\mathcal{P}$
used to construct the $\varXi$-coalescent in $\mathcal{P}_1$ and $\mathcal
{P}_2$. We also use the notation from Lemma \ref{lem:PPPof1} and its proof.
The block of 1 in the $\varXi$-$n$-coalescent for any $n\in\nz$ can only merge
at times $t$ for which there exists a Poisson point
$(T,(K_i)_{i\in\nz})\in\mathcal{P}_1$. Lemma \ref{lem:PPPof1} states
that the
set of times $T$ forms a homogeneous Poisson process with rate $\mu_{-1}$.
This shows that potential jump times are separated by countably many
independent ${\rm Exp}(\mu_{-1})$ random variables. Kingman's correspondence
yields that $f_1$ exists almost surely at each potential jump time.
To see this, observe that even though the partition of $\nz$ induced
by the Poisson construction is not exchangeable, the partition on
$\nz\setminus\{1\}$ is, and the asymptotic frequencies of the former
and the latter coincide. Since $f_1$ is by definition constant between
these jump points, $f_1$ has c\`adl\`ag paths almost surely. Since any
blocks change by mergers, $f_1$ is increasing.

The value of $f_1$ at 0 follows by definition. Since $\varPi$ stays
infinite (see Lemma \ref{lem:CDI}), at each $P\in\mathcal{P}_1$
infinitely many blocks can potentially merge. Lemma \ref{lem:PPPof1}
shows that the indicators of whether blocks present immediately before
$P$ merge with the block of 1 are exchangeable($X$) indicators with
$X>0$ almost surely. Then, Lemma \ref{lem:exch_ind} ensures that a
positive fraction of them almost surely does, causing $f_1$ to jump
(since a positive fraction of merging blocks has positive frequency).
Thus, every Poisson point leads to a merger almost surely, which shows
that $f_1$ jumps at all potential jump times described above. Since,
for all $t$, either $S_t>0$ or non-dust blocks not including 1 exist
(having asymptotical frequency $>0$), $f_1(t)<1$ for all $t\geq0$.

We consider the jump chain of $f_1$. Set $X_1:=f_1[1]$. Since $f_1[k]\in
(0,1)$ for all $k\in\nz$ and $f_1$ increases,
$f_1[k+1]=f_1[k]+(1-f_1[k])X_{k+1}$ for $X_{k+1}\in(0,1)$. Iterating
this yields $f_1[k]=\sum_{i=1}^{k}X_i\prod_{j=1}^{i-1}(1-X_j)$ for
$k\geq2$. The properties of $(X_k)_{k\in\nz}$ follow from the Poisson
construction and Lemma \ref{lem:PPPof1}. Consider the blocks present at
time $T_k-$, where the $k$th Poisson point of $\mathcal{P}_1$ is
$P_k=(T_k,(K^{(k)}_i)_{i\in\nz})$. The block with least element $i$
merges with the block of 1 if $K^{(k)}_i=K^{(k)}_1$. Consider the $k$th
Poisson point at time $T_k$. $X_k$ gives the fraction of the asymptotic
frequency of non-singleton blocks and singleton blocks at time $T_k-$,
i.e. the fraction of $1-f_1(T_{k-})$, that is merged with the block of
1 at $T_k$. Denote $L^{(k-)}_i:=1_{\left\{\mbox{$\{i\}$ is a block at
$T_k-$}\right\}}$. Then, recording the asymptotic frequencies of merged
non-singleton and singleton blocks,
\[
X_k=\frac{1}{1-f_1(T_k-)} \Biggl(\sum_{i\geq2}
1_{\{K^{(k)}_1=K^{(k)}_i\}
}f_i(T_k-)+\lim_{n\to\infty}
\frac{1}{n}\sum_{i=2}^{n}1_{\{
K^{(k)}_1=K^{(k)}_i\}}L^{(k-)}_i
\Biggr).
\]
Since by construction, $\varPi_{T_k-}\setminus\{1\}$ is an
exchangeable partition of $\nz\setminus\{1\}$, $(L^{(k-)}_i)_{i\in\nz}$
are exchangeable($S_{t-}$) indicators with $S_{t-}=1-\sum_{i=1}^\infty
f_i(T_k-)$. Recall that Lemma \ref{lem:PPPof1} tells us that $(1_{\{
K^{(k)}_1=K^{(k)}_i\}})_{i\geq2}$ are exchangeable($X'$) indicators
with $X'\stackrel{d}{=}Q$.
$(K^{(k)}_i)_{i\in\nz}$ is independent from $(\varPi_t)_{t<T_k}$, since
the Poisson points of $\mathcal{P}_1$ are i.i.d., so Lemma \ref
{lem:exch_ind} c) and a) show
\begin{equation}
\label{eq:Xk} X_k\stackrel{a.s.} {=} \sum
_{i\geq2} 1_{\{K^{(k)}_1=K^{(k)}_i\}}\frac
{f_i(T_k-)}{1-f_1(T_k-)}+X'
\frac{1-\sum_{i=1}^\infty f_i(T_k-)}{1-f_1(T_k-)}.
\end{equation}
The independence of $(K^{(k)}_i)_{i\in\nz}$ from $(\varPi_t)_{t<T_k}$ is
also crucial for the next two equations.
Compute, with $P(K^{(k)}_1=K^{(k)}_i)=E(X')=\gamma$ for $i\in\nz$,
\begin{align*}
E(X_k)&=\sum_{i\geq2}P\bigl(K^{(k)}_1=K^{(k)}_i
\bigr) E \biggl(\frac{f_i(T_k-)}{1-f_1(T_k-)} \biggr)+E\bigl(X'\bigr)E \biggl(
\frac{1-\sum_{i=1}^\infty f_i(T_k-)}{1-f_1(T_k-)} \biggr)
\\
&= \gamma E \biggl(\frac
{1-f_1(T_k-)}{1-f_1(T_k-)} \biggr)=\gamma.
\end{align*}
Analogously, for $l<k$, $X_l$ only depends on Poisson points $P_1,\ldots
,P_l$, so
\begin{align*}
E(X_kX_l)&= E \biggl( \sum
_{i\geq2} 1_{\{K^{(k)}_1=K^{(k)}_i\}}\frac
{f_i(T_k-)}{1-f_1(T_k-)}X_l+X'
\frac{1-\sum_{i=1}^\infty
f_i(T_k-)}{1-f_1(T_k-)}X_l \biggr)
\\
&=\gamma E \biggl(\frac
{1-f_1(T_k-)}{1-f_1(T_k-)}X_l \biggr)=
\gamma^2,
\end{align*}
showing that $X_k,X_l$ are uncorrelated.
An analogous computation shows that\break $E(\prod_{i\in\{l_1,\ldots,l_m\}
}X_{l_i})=\prod_{i\in\{ l_1,\ldots,l_m\}}E(X_{l_i})$ for distinct
$l_1,\ldots,l_m\in\nz$. With this,
\[
E\bigl(f_1[k]\bigr)=\sum_{i=1}^k
E(X_i)\prod_{j=1}^{i-1}
\bigl(1-E(X_j)\bigr)=\sum_{i=1}^k
\gamma(1-\gamma)^{i-1}=1-(1-\gamma)^k.
\]
To prove $\lim_{t\to\infty}f_1(t)=1$ almost surely,
observe that $f_1$ is bounded and increasing, thus $\lim_{t\to\infty}
f_1(t)$ exists.
Monotone convergence and $\lim_{t\to\infty}E(f_1(t))=\break\lim_{k\to\infty
}E(f_1[k])=1$ show the desired.
Note that $(X_k)_{k\in\nz}$ is in general
neither independent nor identically
distributed, see %Example
Section~\ref{example:Xk}.
\end{proof}

%\newpage

\begin{proof}[Proof of Corollary \ref{cor:mcs}]
By the Poisson construction the block
of 1 for $\varPi^{(n)}$ can only merge at times given by Poisson points in
$\mathcal{P}_1$. Consider $(T_1,(K^{(1)}_i)_{i\in\nz})\in\mathcal
{P}_1$. While $T_1$ is the time of the first jump of $f_1$ (see the
proof of Theorem \ref{thm:blockof1Xi}), there does not necessarily need
to be a merger of $\{1\}$ in the $n$-coalescent $\varPi^{(n)}$, if we have
$K^{(1)}_1\neq K^{(1)}_i$ for the least elements $i$ of all other
blocks of $\varPi^{(n)}$ immediately before $T_1$. However, Lemma \ref
{lem:PPPof1} shows that $(1_{\{K_1^{(1)}= K_i^{(1)}\}})_{i\geq2}$ are
exchangeable indicators. The mean $n^{-1}\sum_{i=2}^n 1_{\{K_1^{(1)}=
K_i^{(1)}\}}$, as argued in the proof of Theorem \ref{thm:blockof1Xi},
converges to an almost surely positive random variable for $n\to\infty
$. As shown in Lemma \ref{lem:CDI}, any $\varXi$-coalescent with dust has
infinitely many blocks almost surely before $T_1$. Thus, there exists
$N$, a random variable on $\nz$, so that 1 is also merging at time
$T_1$ in $\varPi^{(n)}$ for $n\geq N$ almost surely. This yields
$\lim_{n\to\infty}n^{-1}M_n=\lim_{n\to\infty}n^{-1}|B_1(T_1)\cap
[n]|=f_1(T_1)=f_1[1]$ almost surely. All further claims follow from
Theorem \ref{thm:blockof1Xi}.
\end{proof}
\begin{remarksingle}
Let $Q^{(n)}$ be the number of blocks merged at the first collision of
the block of 1 in a $\varLambda$-$n$-coalescent with dust. \cite
[1.4]{siri2016asymptotics} shows that $n^{-1}Q^{(n)}$ converges in
distribution. We argue that this convergence also holds in $L^p$ for
all $p>0$ and, for simple $\varLambda$-$n$-coalescents, almost surely.

The proof of Corollary \ref{cor:mcs} shows that $(T_1,(K^{(1)}_i)_{i\in
\nz})\in\mathcal{P}_1$ causes the first merger in the $n$-coalescent
for $n$ large enough (almost surely, but since $n^{-1}Q^{(n)}\in[0,1]$
for all $n$, convergence in $L^p$ is not affected by the null set
excluded). Split $Q^{(n)}$ into $Q^{(n)}_0$, the number of
non-singleton blocks and $Q^{(n)}_1$, the number of singleton blocks
merged at $T_1$. For the limit, we can ignore the non-singleton blocks
merged. To see this, recall $Q^{(n)}_0\leq K_n$, where $K_n$ is the
total number of mergers for the $\varLambda$-$n$-coalescent, since a
non-singleton block has to be the result of a merger. \cite[Lemma
4.1]{freund2009number} tells us that $n^{-1}K_n\to0$ in $L^1$ for $n\to
\infty$ for $\varXi$-coalescents with dust. This shows that the
$L^1$-limit of $n^{-1}Q^{(n)}$ is the same as of the one of
$n^{-1}Q^{(n)}_1$. $n^{-1}Q^{(n)}_1$ already appeared in the part of
the proof of Theorem \ref{thm:blockof1Xi} leading to Eq.~\eqref{eq:Xk},
its limit almost surely exists and equals $X'\frac{1-\sum_{i=1}^\infty
f_i(T_1-)}{1-f_1(T_1-)}$. Since $n^{-1}Q^{(n)}_1$ is bounded in
$[0,1]$, it also converges in $L^p$, $p>0$. So $n^{-1}Q^{(n)}$
converges in $L^1$. Since it is bounded in $[0,1]$ it also converges in
$L^p$, $p>0$. For simple $\varXi$-$n$-coalescents, \cite[Lemma
4.2]{Freund2012} shows $n^{-1}K_n\to0$ almost surely, so in this case
the steps above ensure also almost sure convergence of $n^{-1}Q^{(n)}$.
\end{remarksingle}

\section{The block of 1 in simple \texorpdfstring{$\varLambda
$}{Lambda}-coalescents -- proofs and remarks}\label{sec:simple}
\begin{proof}[Proof of Proposition \ref{prop:f1simple}]
Let $\mathcal{P}:=(P_i)_{i\in\nz}$ be the coin probabilities coming
from the Poisson process used to construct the simple $\varLambda
$-coalescent $\varPi$ as described in Section~\ref{sect:PoissC}. As shown
in the proof of Theorem \ref{thm:blockof1Xi}, the Poisson point
belonging to $P_C$ \querymark{Q3}where $1$ first throws `heads' in the
Poisson construction is the Poisson point where $f_1$ jumps for the
first time. We have $P(C=k|\mathcal{P})=P_k\prod_{i=1}^{k-1} (1-P_i)$.
Integrating the condition and using the independence of $(P_i)_{i\in\nz
}$ as well as $E(P_1)=\alpha$ (see Lemma~\ref{lem:simplePPP}), we see
that $C$ is geometrically distributed with parameter $\alpha$.

To describe $f_1[1]$ at the $C$th merger (Poisson point), recall that
the restriction $\varPi_{-1}$ of $\varPi$ to $\nz\setminus\{1\}$ has the
same asymptotic frequencies as $\varPi$. Thus, we can see $f_1[1]$ as the
asymptotic frequency of the newly formed block of $\varPi_{-1}$ at the
time of the Poisson point $P_C$. This follows since $\varPi_{-1}$ has
infinitely many blocks before (see Lemma \ref{lem:CDI}) and then, as in
the proof of Theorem \ref{thm:blockof1Xi}, there will be a newly formed
block of $\varPi_{-1}$ at the $C$th Poisson point (and the unrestricted
block in $\varPi$ includes 1).

We consider $\varPi_{-1}$ at the $k$th Poisson point with coin probability
$P_k$. For $\{i\}\in\nz\setminus\{1\}$ to remain a (singleton) block
and not be merged for the first $k-1$ mergers and then to be merged at
the $k$th, we need $\prod_{j\in[k-1]}(1-K^{(j)}_i)=1$ and
$K^{(k)}_i=1$. $(1_{\{\prod_{j\in[k-1]}(1-K^{(j)}_i)=1,K^{(k)}_i=1\}
})_{i\in\nz}$ are exchangeable($P_k\prod_{j\in[k-1]}(1-P_j)$)
indicators. Let
\[
\mathcal{S}_k= \biggl\{i\in\nz\setminus\{1\} : \prod
_{j\in
[k-1]}\bigl(1-K^{(j)}_i
\bigr)=1,K^{(k)}_i=1 \biggr\}
\]
be the set of $i\in\nz\setminus\{1\}$ whose first merger is the $k$th
overall merger. We call $\mathcal{S}_k$ the $k$th singleton set
(of $\varPi_{-1}$). From the strong law of large numbers for exchangeable
indicators, see Lemma \ref{lem:exch_ind}a), we directly have that
$\mathcal{S}_k$ has asymptotic frequency $P_k\prod_{j\in[k-1]}(1-P_j)$
almost surely.

Now, consider the asymptotic frequency $f^*[k]$ of the newly formed
block at the $k$th merger of $\varPi_{-1}$. By construction, there is only
one newly formed block at each merger. $\mathcal{S}_k$ is a part of the
newly formed block. Any other present block with more than two elements
(non-singleton block) is merged if and only if its indicator
$K^{(k)}_i=1$ (we order by least elements). For $k=1$, the newly formed
block is $\mathcal{S}_1$. For $k=2$, it is either $\mathcal{S}_2$ or
$\mathcal{S}_1\cup\mathcal{S}_2$, if the coin of the the block
$\mathcal{S}_1$ formed in the first merger comes up `heads'.

Applied successively, this shows that the newly formed block at the
$k$th merger consists of a union of a subset of the singleton sets
$(\mathcal{S}_{k'})_{k'<k}$ and the set $\mathcal{S}_k$. For its
asymptotic frequency, we have
\begin{equation}
\label{eq:newblock} f^*[k]=\sum_{i=1}^k
B^{(k)}_i P_i\prod
_{j\in[i-1]}(1-P_j)>0,
\end{equation}
where the $B^{(k)}_i$, $i\in[k]$, are non-independent Bernoulli
variables which are 1 if the $i$th singleton set $\mathcal{S}_i$ is a
part of the newly formed block at the $k$th merger of $\varPi_{-1}$.

If $\varLambda(\{1\})>0$, $P_k=1$ is possible. In this case, at the $k$th
Poisson point all remaining singletons form $\mathcal{S}_k$ and all
blocks present at merger $k-1$ merge with $\mathcal{S}_k$. There are no
mergers at Poisson points $P_{l}$, $l>k$, so we do not consider
Eq.~\eqref{eq:newblock} for $l>k$.

We have $f_1[1]=f^*[C]$. Given $\mathcal{P}$, $(f^*[k])_{k\in\nz}$ is
independent of $C$. Thus, Eq.~\eqref{eq:dist_C} is implied by Eq.~\eqref
{eq:newblock}.

Assume $\varLambda(\{1\})=0$. For $(B^{(k)}_i)_{k\in\nz,i\in[k]}$, we have
$B_k^{(k)}=1$ for all $k\in\nz$ since the $k$th singleton set is formed
at the $k$th Poisson point and is a part of the newly formed block. The
coins thrown at the $k$th Poisson point to decide whether other
singleton sets $\mathcal{S}_i$, $\mathcal{S}_{j}$ with $i,j<k$ are also
parts of the newly formed block are either independent given $\mathcal
{P}$ when they are in different blocks, or identical when they are in
the same block. The set $\mathcal{S}_i$ uses the coin of the block
newly formed at the $i$th merger. Let $I(i)$ be the Poisson point at
which this block merges again (and $\mathcal{S}_i$ with it). At the
$I(i)$th Poisson point and for all further Poisson points indexed with
$j\geq I(i)$, we have $B_i^{(j)}=B_{I(i)}^{(j)}$, since the singleton
sets $S_i$ and $S_{I(i)}$ are in the same block for mergers $j\geq I(i)$.

The property (i) of $I(i)$ in the proposition follow directly from its
definition as the minimum number of coin tosses until the first comes
up `heads'. The property (ii) is just integrating (i) and using that
$(P_i)_{i\in\nz}$ are i.i.d.~with $E(P_1)=\alpha$ (see Lemma \ref
{lem:simplePPP}), the conditional independence is the conditional
independence of coin tosses of distinct blocks from the Poisson
construction. To see Eq.~\eqref{eq:jointdistbi}, observe that $\mathcal
{S}_i$ for $i<j$ is a part of the newly formed block at the $j$th
merger of the $\varLambda$-coalescent ($i\in J$) if and only if $I(i)\in
J$. If $I(i)\in J$, either we have $I(i)=j$, so $\mathcal{S}_i$ is
merged for the first time after it has been formed at the $j$th merger,
or we have that $I(i)<j$ which means that it has already merged with at
least one other singleton set and that, as parts of the same block,
they both again merged at the $j$th merger. If $I(i)\notin J$, the
singleton set $\mathcal{S}_i$ neither merges at the $j$th merger for
the first time after being formed nor merges with any other singleton
set before that is then merging at the $j$th merger, so $\mathcal{S}_i$
is not a part of the newly merged block at the $j$th merger.

If $\varLambda(\{1\})=0$, the arguments hold true for all $j\in\nz$.
If $\varLambda(\{1\})>0$ this holds true for all $i,j\leq
K:=\min_{k\in\nz}\{P_k=1\}(<\infty\mbox{ almost surely})$, where
all singleton sets merge and $f^{*}[K]=1$. However, in this case
$C\leq K$, so we still can establish Eq.~\eqref{eq:dist_C}.
\end{proof}

\begin{remark}\label{remark1}\
\begin{itemize}
\item$(I(i))_{i\in\nz}$ is useful to construct the asymptotic
frequencies of the $\varLambda$-coales\-cent. Given $\mathcal{P}$, at the
$k$th merger, there are the singleton sets $(\mathcal{S}_j)_{j\in[k]}$
with almost sure frequencies $P_j\prod_{i\in[j-1]}(1-P_i)$ which were
already formed in the $k$ collisions, and unmerged singleton blocks
with frequency $\prod_{i\in[k]}(1-P_i)$. Using $(I(i))_{i\in[k]}$, we
can indicate which singleton sets form a block. $\mathcal{S}_i$ is a
single block if $I(i)>k$, if $I(i)\leq k$ it is a part of a block where
$\mathcal{S}_{I(i)}$ is also a part of. This can be seen as a discrete
version of the construction of the $\varLambda$-coalescent from the
process of singletons as described in \cite[Section~6.1]{gnedin2014lambda}
\item The variables $(I(i))_{i\in\nz}$ are useful to express other
quantities of the $\varLambda$-coa\-lescent. For instance, the number of
non-singleton blocks in a simple $\varLambda$-coa\-lescent at the $k$th
merger is given by $k-\sum_{i\in[k-1]} 1_{\{I(i)\leq k\}}$.
\end{itemize}
\end{remark}

To prove Proposition \ref{cor:Dirac}, we need the following result.\vadjust{\goodbreak}
\begin{lemma}\label{lem:unique}
For $p\in[\frac{1}{2},1)$, each $x\in\mathcal{M}_p$ from Eq.~\eqref
{eq:unique} has a unique representation in $\mathcal{M}_p$.
\end{lemma}
\begin{proof} We adjust the proof of \cite[Theorem 7.11]{amann2005analysis}.
Assume that $x\in\mathcal{M}_p$ has two representations
$x=\sum_{i\in\nz}b_ipq^{i-1}=\sum_{i\in\nz}b'_ipq^{i-1}$ with
$b_i\neq b'_i$ for at least one $i$. Let $i_0$ be the smallest integer
with $b_{i_0}\neq b'_{i_0}$. Without restriction, assume
$b_{i_0}-b'_{i_0}=1$. Then,
\[
0=\sum_{i\in\nz} b_i pq^{i-1}-\sum
_{i\in\nz}b'_ipq^{i-1}
=pq^{i_0-1}+\sum_{i>i_0}\bigl(b_i-b'_i
\bigr)pq^{i-1}.
\]
Thus, $pq^{i_0-1}=\sum_{i>i_0}(b'_i-b_i)pq^{i-1}<\sum_{i>i_0}pq^{i-1}
=pq^{i_0}$, simplifying to $p<q$, in contradiction to the assumption
$p\geq\frac{1}{2}$.
\end{proof}

\begin{proof}[Proof of Proposition \ref{cor:Dirac}]
From Eq.~\eqref{eq:newblock} we see that $f_1$ only takes values in
$\mathcal{M}_p$, since $P_k=p$ for all $k\in\nz$ and $C<\infty$ almost surely.
Recall the definition of the singleton sets $\mathcal{S}_i$ and their
properties from the proof of Proposition \ref{prop:f1simple}. The
asymptotic frequency of $\mathcal{S}_i$ is $pq^{i-1}$ almost surely.
Lemma \ref{lem:unique} ensures that there is a unique representation
$f_1[l]=x=\sum_{i=1}^\infty b_ipq^{i-1}$ in $\mathcal{M}_p$, let $J:=\{
i\in\nz:b_i=1\}$ and $j:=\max J$. This means that $f_1[l]=x$ is
equivalent to that the block of 1 at its $l$th jump consists of the
union of all $\mathcal{S}_i$ with $i\in J$ and 1. This also shows that
the $l$th jump of $f_1$ is at the $j$th jump of the Dirac coalescent,
since if $f_1$ jumps at the $k$th merger of the Dirac coalescent, the
newly formed block includes $\mathcal{S}_k$.

Since $P_i=p$ for all $i\in\nz$, we have $\alpha=p$ and Eq.~\eqref
{eq:dist_C} simplifies to
$f_1[1]=\sum_{i=1}^C B_i^{(C)} pq^{i-1}$, where $C\stackrel
{d}{=}Geo(p)$ is independent from $(B_i^{(k)})_{k\in\nz,i\in[k]}$. The
latter fulfil
\begin{equation}
\label{eq:distf1_1_b_i} P\bigl(B_i^{(j)}=b_i\ \forall\
i\in[j-1]\bigr)= \prod_{i\in J\setminus\{j\}
}P(Y+i\in J)\prod
_{i\in[j]\setminus J} P(Y+i\notin J)
\end{equation}
with $Y\stackrel{d}{=}Geo(p)$, since the joint distribution in
Proposition \ref{prop:f1simple} again simplifies, we can ignore the
conditioning and
$I(i)-i\stackrel{d}{=}Geo(p)$ for all $i\in\nz$.

Since $f_1[1]=x$ uniquely determines the values of $C$ and
$(B_i^{(C)})_{i\in[C]}$, we have
\begin{equation}
\label{eq:h1} P\bigl(f_1[1]=x\bigr)=P(C=j)P\bigl(B_1^{(j)}=b_1,
\ldots,B_{j-1}^{(j)}=b_{j-1}\bigr),
\end{equation}
which shows Eq.~\eqref{eq:distXDirac} when we insert the distributions
expressed in terms of their geometric distributions.

In order to verify that the jump chain $(f_1[i])_{i\in\nz}$ is
Markovian, we
show that $f_1[1],\ldots,f_1[l]$ does not contain
more information on $f_1[l+1]$ than $f_1[l]$ does. Without restriction,
assume that the $l$th jump $f_1[l]$ of $f_1$ takes place at the $k$th
jump of the Dirac coalescent. Then, $f_1[l+1]$ is constructed from the
blocks present after the $k$th merger. For each subsequent Poisson
point $P_{k+1},\ldots$, blocks present are merged if their respective
coins come up `heads' until (and including), at $P_{k'}$, the coin of
the block of 1 comes up `heads' for the first time since $P_k$. Thus,
only information about the block partition at merger $k$ can change the
law of the next jump. $f_1[l]=x$ gives the information which singleton
sets $\mathcal{S}_1,\ldots,\mathcal{S}_k$ are parts of the block of 1
at merger $k$ of $\varPi$ and which are not. $f_1[l]=x$ contains no
information about how the other singleton sets, $\mathcal{S}_i$ with
$b_i=0$, are merged into blocks at collisions before $k$ apart from
that it tells us that $B^{(j)}_i=0$ for $j\in J$ and $i\notin J$, which
means that all $\mathcal{S}_i$ with $i\notin J$ did not merge at the
$j$th collisions, $j\in J$. This is due to that any $\mathcal{S}_i$
with $B^{(j)}_i=1$ would merge with the newly formed block at merger
$j$ and thus would be in a block with $\mathcal{S}_j$ and also in the
block of $1$ at merger $k$. However, analogously we see that knowing
$f_1[1],\ldots,f_1[l]$ does not give any additional information about
the block structure at the $k$th merger, but only how the set of
$\mathcal{S}_i$ which are in the block of 1 at merger $k$ behaved at
the earlier mergers $J$. Thus, $(f_1[l])_{l\in\nz}$ is Markovian.
However, $(f_1(t))_{t\geq0}$ is
not Markovian. In order to see this consider, for $0<t_0<t_1<t_2$,
\begin{align*}
p(t_2,t_1,t_0):=& P\bigl(f_1(t_2)=p+pq^2
\vert f_1(t_1)=p,f_1(t_0)=0\bigr)
\\
=& \frac{P(f_1(t_2)=p+pq^2, f_1(t_1)=p,f_1(t_0)=0)}{P(f_1(t_1)=p,f_1(t_0)=0)}.
\end{align*}
We will show that $p(t_2,t_1,t_0)$ depends on $t_0$, which shows that
$f_1$ is not Markovian.

We can express all events in terms of the independent waiting times for
Poisson points, i.e. the successive differences between the first
component $T$ of the Poisson points $(T,(K_i)_{i\in\nz})\in\mathcal
{P}$. Here, we use the split of the Poisson points into the independent
Poisson point processes $\mathcal{P}_1$ and $\mathcal{P}_2$ from Lemma
\ref{lem:PPPof1}. \querymark{Q4}The waiting times between points in
$\mathcal{P}_1$ are ${\rm Exp}(\mu_{-1})$-distributed, the waiting times between points in
$\mathcal{P}_2$ are ${\rm Exp}(\mu_{-2}-\mu_{-1})$-distributed, see Lemma \ref{lem:PPPof1} and
Remark \ref{rem:PPP}. We will relabel $\tau=\mu_{-1}$ and $\rho=\mu
_{-2}-\mu_{-1}$ for a clearer type face. Let $T_1,T_2,\ldots$ be the
waiting times between points in $\mathcal{P}_1$ and $T'_1,T'_2,\ldots$
be the waiting times between points in $\mathcal{P}_2$. All waiting
times are independent one from another. We recall that for $T\stackrel
{d}{=}{\rm Exp}(\alpha)$, $P(T>a)=e^{-\alpha a}$ and $P(T\in
(a,a+b])=e^{-\alpha a}(1-e^{-\alpha b})$ for $a,b\geq0$.

The event $\{f_1(t_1)=p,f_1(t_0)=0\}$ means that the first jump of
$f_1$ adds the singleton set $\mathcal{S}_1$ at a time in $(t_0,t_1]$.
Thus, there has to be only a single point of $\mathcal{P}_1$ with first
component $T_1\leq t_1$ and the smallest time $T'_1$ of points of
$\mathcal{P}_2$ has to be greater than $T_1$. We compute, conditioning
on $T_1$ for the third equation,
\begin{align*}
P\bigl(f_1(t_1)=p,f_1(t_0)=0
\bigr)& =P\bigl(t_0<T_1\le t_1<T_1+T_2,T_1<T_1'
\bigr) %=P(T_1\in(t_0,t_1],T_1+T_2>t_1,T_1'><T_1)
\\
&= \int_{t_0}^{t_1}P(T_2>t_1-x)P
\bigl(T'_1>x\bigr)\tau e^{-\tau x}dx \\
&= \int
_{t_0}^{t_1} e^{-\tau(t_1-x)}e^{-\rho x}\tau
e^{-\tau x}dx
\\
&= \frac
{\tau}{\rho}e^{-\tau t_1}\int_{t_0}^{t_1}
\rho e^{-\rho x}dx = \frac
{\tau}{\rho}e^{-\tau t_1}\bigl(e^{-\rho t_0}-e^{-\rho t_1}
\bigr).
\end{align*}
Analogously, we compute (by conditioning on $T_1,T_2$ for the second equality)
\begin{align*}
& P\bigl(f_1(t_2)=p+pq^2,
f_1(t_1)=p,f_1(t_0)=0\bigr)
\\
%=& P(T_1\in(t_0,t_1],T_1+T_2\in(t_1,t_2], T_1+T_2+T_3>t_2,T_1'
%\in(T_1,T_1+T_2],T_2'> T_1+T_2)\\
&\quad= P\bigl(t_0<T_1\le
t_1<T_1+T_2\le t_2<T_1+T_2+T_3,T_1<T_1'
\le T_1+T_2<T_2'\bigr)
\\
&\quad= \int_{t_0}^{t_1} \int_{t_1-x}^{t_2-x}
P\Biggl(T_3\,{>}\,t_2\,{-}\,x\,{-}\,y,T'_1\,{\in}\,
\bigl(x,x\,{+}\,y],T'_2\,{>}\,x\,{+}\,y\bigr) \tau^2
e^{-\tau x}e^{-\tau y}dy\ dx
\\
&\quad= \tau^2\int_{t_0}^{t_1} \int
_{t_1-x}^{t_2-x} e^{-\tau
(t_2-x-y)}e^{-\rho x}
\bigl(1-e^{-\rho y}\bigr)e^{-\rho(x+y)}e^{-\tau x}e^{-\tau
y}dy\
dx
\\
&\quad= \frac{\tau^2}{\rho}e^{-\tau t_2} \biggl[\rho^{-1}
\bigl(e^{-\rho
t_1}\,{-}\,e^{-\rho t_2}\bigr) \bigl(e^{-\rho t_0}\,{-}\,e^{-\rho t_1}
\bigr)\Biggr)\,{-}\,\frac{1}{2}\bigl(e^{-2\rho
t_1}\,{-}\,e^{-2\rho t_2}\bigr)
(t_1-t_0) \biggr].
\end{align*}

Taking the ratio shows that
\begin{align*}
p(t_2,t_1,t_0)= \frac{\tau}{\rho}
e^{-\tau(t_2-t_1)}\bigl(e^{-\rho
t_1}-e^{-\rho t_2}\bigr)-\underbrace{
\frac{\tau}{2}e^{-\tau(t_2-t_1)}\frac
{e^{-2\rho t_1}-e^{-2\rho t_2}}{e^{-\rho t_0}-e^{-\rho t_1}}}_{\neq0}(t_1-t_0)
\end{align*}
depends on $t_0$, so $f_1$ is not Markovian.
\end{proof}
\begin{remarksingle}
Our proof of Proposition \ref{cor:Dirac} relies on the unique
representation in $\mathcal{M}_p$. This means that it also holds true
for all $p\in(0,2^{-1})$ where each $x\in\mathcal{M}_p$ has a unique
representation in $\mathcal{M}_p$, e.g. for all transcendental $p$.
%that are transcendental.
If the representation is not unique, Eq. \eqref{eq:distf1_1_b_i} is
still correct, but the right side of Eq. \eqref{eq:h1} does not show
$P(f_1[1]=x)$. Instead, the latter shows the contribution to
$P(f_1[1]=x)$ from the paths of $f_1$ which fulfil
$C=j,B_1^{(j)}=b_1,\ldots,B_{j-1}^{(j)}=b_{j-1}$ (recall that
$j,b_{1},\ldots,b_{j-1}$ depend on the representation of $x$).
Moreover, $P(f_1[1]=x)$ then is the sum over
$P(C=j)P(B_1^{(j)}=b_1,\ldots,B_{j-1}^{(j)}=b_{j-1})$ for the tuples
$j,b_{1},\ldots,b_{j-1}$ coming from the different representations of
$x$ (the sets of paths are disjoint if the parameter sets $(j,b_1,\ldots
,b_{j-1})$ differ).
Since the proof of our results on the Markov property of both $f_1$
and its jump chain also rely on the unique representation of $x$ (to
read off which blocks merged when), the proof does not extend if $p$
does not allow a unique representation of $x$.
\end{remarksingle}
\section{Example}\label{example:Xk}
We provide a concrete example showing that the random variables
$(X_k)_{k\in\nz}$ from Theorem \ref{thm:blockof1Xi} are, in general,
neither independent nor identically distributed.

Choose $\varLambda=\delta_\frac{1}{2}$ and consider $f_1$ in the
corresponding $\varLambda$-coalescent. Recall that $f_1[l]=x\in\mathcal
{M}_{\frac{1}{2}}$ already fixes which singleton sets $\mathcal{S}_k$
are parts of the block of 1 at its $l$th merger and which are not.
First, assume
$f_1[1]=X_1=\frac{5}{8}=\frac{1}{2}+\frac{1}{2^3}\in\mathcal{M}_{\frac{1}{2}}$,
which means that the coin of 1 comes up `heads' for the first time
at the third Poisson point and the block of 1 is
$\mathcal{S}_1\cup\mathcal{S}_3$, while $\mathcal{S}_2$ is a
block of its own (an event happening with probability $>0$).
Assume further $f_1[2]=\frac{11}{16}=\frac{5}{8}+\frac{1}{16}$.
This sets
$X_2=(f_1[2]-f_1[1])/(1-X_1)=\frac{1}{6}\notin\mathcal{M}_{\frac{1}{2}}$.
We read off that the coin of the block of 1 also comes up `heads'
at the fourth collision, where the block of 1 merges with
$\mathcal{S}_4$. We also see that the coin of the only other block
$\mathcal{S}_2$ comes up `tails'. We thus have, since we throw
fair coins, $P(X_2=\frac{1}{6}\vert
X_1=\frac{5}{8})=P(f_1[2]=\frac{11}{16}\vert
f_1[1]=\frac{5}{8})=\frac{1}{4}$. Since
$X_1=f_1[1]\in\mathcal{M}_{\frac{1}{2}}$ for any realisation,
$X_1$ and $X_2$ have different distributions. To see also
non-independence, consider $f_1[1]=X_1=\frac{1}{2}$ (coin of 1
comes up `heads' at first coin toss, block of 1 is
$\mathcal{S}_1$, occurs with probability $\frac{1}{2}$). In this
case $P(X_2=\frac{1}{6}|X_1=\frac{1}{2})=0$, since
$f_1[2]=X_1+(1-X_1)X_2=\frac{7}{12}\notin\mathcal{M}_{\frac{1}{2}}$.
% \end{example}

\begin{acknowledgement}[title={Acknowledgements}]
Jason Schweinsberg and an anonymous reviewer pointed
out that $f_1$ is not Markovian for Dirac coalescents, the latter also
stated the uniqueness condition in Lemma \ref{lem:unique}. Matthias
Birkner remarked that the uniqueness in Lemma \ref{lem:unique} extends
to transcendental $p$. We thank them and all reviewers for their
constructive comments leading to an improvement of the manuscript.

F. Freund was funded by the grant FR 3633/2-1 of the German Research
Foundation (DFG) within the priority program 1590 ``Probabilistic
Structures in Evolution''.
\end{acknowledgement}

%\bibliographystyle{vmsta-mathphys}
%\bibliography{bib}
%

\end{document}